\documentclass[10pt]{amsart}

\usepackage{amssymb,latexsym,amscd,amsmath,amsthm}
\usepackage{epsfig}
\usepackage{psfrag}
\usepackage{graphicx}
\usepackage{color}
\usepackage{geometry}
\usepackage{multirow}
\usepackage{hyperref}
\usepackage{comment}
\usepackage{palatino}

\usepackage[all]{xy}
\xyoption{all}
\xyoption{arc}
\xyoption{poly}

\usepackage{tikz}
\usetikzlibrary{backgrounds,calc}

\usepackage[hang,small,bf]{caption}
\newtheorem{thm}{Theorem}[section]
\newtheorem{lem}[thm]{Lemma}

\newtheorem{cor}[thm]{Corollary}

\theoremstyle{plain}
\newtheorem{main}{Theorem}

\theoremstyle{definition}

\newtheorem{rem}[thm]{Remark}

\newtheorem{eg}[thm]{Example}

\newcommand{\Z}{{\mathbb{Z}}}

\newcommand{\C}{{\mathbb{C}}}

\newcommand{\SU}{\operatorname{SU}}

\newcommand{\syp}{\operatorname{Sp}}

\newcommand{\U}{\operatorname{U}}

\renewcommand{\lim}[1]{\mathop{\underset{#1} {\underset \longleftarrow
{\text{\rm lim}}}}}
\newcommand{\Isom}{\operatorname{Isom}}

\newcommand{\diag}{\operatorname{diag}}

\newcommand{\bq}{/ \hspace{-.12cm} /}

\def\mt{\mapsto}

\newcommand{\g}{\mathfrak{g}}
\def\k{\mathfrak{k}}

\newcommand{\p}{\mathfrak{p}}

\def\x{\times}
\def\wh{\widehat}
\def\hra{\hookrightarrow}
\def\<{\langle}
\def\>{\rangle}
\def\met{\<\, ,\, \>}

\newcommand{\SUU}[2]{\operatorname{S}(\U(#1) \x \U(#2))}

\def\bsm{\begin{smallmatrix}}
\def\esm{\end{smallmatrix}}
\def\bpm{\begin{pmatrix}}
\def\epm{\end{pmatrix}}
\def\bbm{\begin{bmatrix}}
\def\ebm{\end{bmatrix}}
\def\beq{\begin{equation}}
\def\eeq{\end{equation}}

\renewcommand{\geq}{\geqslant}
\renewcommand{\leq}{\leqslant}

\numberwithin{equation}{section}

\newcommand{\spacing}[1]{\renewcommand{\baselinestretch}{#1}\large\normalsize}
\spacing{1.14}

\begin{document}




\title[Totally geodesic embeddings]{A note on totally geodesic embeddings of Eschenburg spaces into Bazaikin spaces}

\author[M.\ Kerin]{Martin Kerin$^\ast$}

\thanks{$^\ast$ This research was carried out as part of SFB 878: Groups, Geometry \& Actions, at the University of M\"unster.}

\address{Mathematisches Institut, Einsteinstr. 62, 48149 M\"unster, Germany}
\email{m.kerin@math.uni-muenster.de}


\subjclass[2010]{53C20}
\keywords{positive curvature, biquotients, Lie groups, totally geodesic embeddings.}


\begin{abstract}
In this note it is shown that every $7$-dimensional Eschenburg space can be totally geodesically embedded into infinitely many topologically distinct $13$-dimensional Bazaikin spaces.  Furthermore, examples are given which show that, under the known construction, it is not always possible to totally geodesically embed a positively curved Eschenburg space into a Bazaikin space with positive curvature.
\end{abstract}

\date{\today}

\maketitle




\normalsize
\thispagestyle{empty}



In dimensions $7$ and $13$ there are two very special families of closed Riemannian manifolds, namely the Eschenburg and Bazaikin spaces.  These are defined as quotients of $\SU(3)$ and $\SU(5)/\syp(2)$ by free, isometric circle actions (see Section \ref{S:EschBaz} and \cite{AW}, \cite{Baz}, \cite{DE}, \cite{Es}, \cite{Ke}, \cite{Zi}), where $\SU(k)$ has been equipped with a left-invariant, right $\U(k-1)$-invariant metric.  In each case there are infinitely many (distinct homotopy types of) family members admitting positive sectional curvature.  Given that there are so few known examples of closed manifolds with positive sectional curvature, these families has been studied extensively (see, for example, \cite{Baz}, \cite{CEZ}, \cite{DE}, \cite{Es}, \cite{Es2}, \cite{FZ}, \cite{GSZ}, \cite{Ke}, \cite{Kr}, \cite{Zi}).  

One particularly intriguing observation, made in \cite{Ta}, is that to each Bazaikin space there can be associated a totally geodesic, embedded Eschenburg space.  In fact, as demonstrated in \cite{DE}, a Bazaikin space generically contains ten mutually distinct, totally geodesic, embedded Eschenburg spaces.  This led to the question: given an Eschenburg space, does there exist a (non-singular) Bazaikin space containing it as a totally geodesic, embedded submanifold?

\begin{main}
\label{Thm A}
For any given Eschenburg space $E^7_{a,b}$, there exist infinitely many mutually non-homotopy equivalent Bazaikin spaces into which $E^7_{a,b}$ can be embedded as a totally geodesic submanifold.
\end{main}

In \cite{DE} it has also been proven that a Bazaikin space is positively curved if and only if each of the ten embedded, totally geodesic Eschenburg spaces it contains are also positively curved.  On the other hand, it is well-known (see \cite{DE}, \cite{Zi}) that all positively curved Aloff-Wallach spaces (see \cite{AW}), that is, the subfamily of homogeneous Eschenburg spaces, and all positively curved cohomogeneity-one Eschenburg spaces (see \cite{GSZ}) can be embedded as totally geodesic submanifolds of positively curved Bazaikin spaces.  This raises the question of whether this is true in general.  However, at least for the known construction of a totally geodesic embedding, there exist counter-examples already in cohomogeneity-two, see Table \ref{tab:fail} in Section \ref{S:curv}.

The article is organised as follows.  In Section \ref{S:Biqs} the basic notation and definitions for biquotients are reviewed.  Section \ref{S:EschBaz} provides a brief summary of Eschenburg and Bazaikin spaces, while Section \ref{S:TGSubmBaz} recalls how to each Bazaikin space there correspond ten totally geodesic, embedded Eschenburg spaces.  In Section \ref{S:Converse} the proof of Theorem \ref{Thm A} is given and, finally, in Section \ref{S:curv} the ability to embed a positively curved Eschenburg space as a totally geodesic submanifold of a positively curved Bazaikin space is discussed.



\section{Biquotients and induced metrics}
\label{S:Biqs}

The following section provides a review of some material from \cite{Es} which establishes the basic ideas and notation which will be used throughout the remainder of the article.

Let $G$ be a compact Lie group, $U \subset G \x G$ a closed subgroup, and let $U$ act (effectively) on $G$ via
\beq
\label{eq:action}        
(u_1, u_2) \star g = u_1 g u_2^{-1}, \ \  g \in G, (u_1, u_2) \in U.
\eeq
The resulting quotient $G \bq U$ is called a {\it biquotient}.  The action (\ref{eq:action}) is free if and only if, for all non-trivial $(u_1, u_2) \in U$, $u_1$ is never conjugate to $u_2$ in $G$.  In the event that the action of $U$ is ineffective, the quotient $G \bq U$ will be a manifold so long as any element $(u_1, u_2) \in U$ which fixes some $g \in G$ fixes all of $G$, that is, $(u_1, u_2)$ lies in the ineffective kernel.  A biquotient is \emph{non-singular} if the action (\ref{eq:action}) is free (modulo any ineffective kernel), and \emph{singular} otherwise.

Let $K \subset G$ be a closed subgroup.  Suppose we have a biquotient $G \bq U$, where $U \subset G \x K \subset G \x G$ and $G$ is equipped with a left-invariant, right $K$-invariant metric $\met$.  Then $U$ acts by isometries on $G$ and therefore the submersion $G \to G \bq U$ induces a metric on $G \bq U$ from the metric on $G$.  We encode this in the notation $(G, \met) \bq U$.  

Now, for $g \in G$ define
        \begin{align*}
            U^g_L &:= \{(g u_1 g^{-1}, u_2) \ | \ (u_1, u_2) \in U \},\\
            U^g_R &:= \{(u_1, g u_2 g^{-1}) \ | \ (u_1, u_2) \in U \},\ \ {\rm and}\\
            \wh U &:= \{(u_2, u_1) \ | \ (u_1, u_2) \in U \}.
        \end{align*}
Then $U^g_L, U^g_R$ and $\wh U$ act freely on $G$, and $G \bq U$ is isometric to $G \bq U^g_L$, diffeomorphic to $G \bq U^g_R$ (isometric if $g \in K$), and diffeomorphic to $G \bq \wh U$ (isometric if $U \subset K \x K$).

In the case of $U^g_L$ this follows from the fact that left-translation $L_g : G \to G$ is an isometry which satisfies $g u_1 g^{-1}( L_g g') u_2^{-1} = L_g (u_1 g' u_2^{-1})$.  Therefore $L_g$ induces an isometry of the orbit spaces $G\bq U$ and $G\bq U^g_L$.  Similarly we find that $R_{g^{-1}}$ induces a diffeomorphism between $G\bq U$ and $G \bq U^g_R$, which is an isometry if $g \in K$.

Consider now $\wh U$.  The actions of $U$ and $\wh U$ are equivariant under the diffeomorphism $\tau : G \to G$, $\tau(g) := g^{-1}$.  That is, $u_1 \tau(g) u_2^{-1} = \tau(u_2 g u_1^{-1})$.  Notice that this is an isometry only if $U \subset K \x K$.  In general $G\bq U$ and $G\bq \wh U$ are therefore diffeomorphic but not isometric.

For completeness we remark that it is, in fact, simple to equip $G$ with a left-invariant, right $K$-invariant metric.  Given $K \subset G$, let $\k \subset \g $ be the corresponding Lie algebras and let $\met_0$ be a bi-invariant metric on $G$.  One can write $\g = \k \oplus \p$ with respect to $\met_0$.  Recall that $G \cong (G \x K)/ \Delta K$ via $(g,k) \mt g k^{-1}$, where $\Delta K = \{(k,k) \mid k \in K\}$ acts diagonally on the right of $G \x K$.  Thus we may define a left-invariant, right $K$-invariant metric $\met$ on $G$ via the Riemannian submersion
        \begin{align*}
            (G \x K, \met_0 \oplus t \met_0 |_\k) &\to (G, \met)\\
            (g,k) &\mt g k^{-1},
        \end{align*}
where $t>0$ and
        \beq
        \label{met1}
            \met = \met_0 |_\p + \lambda \met_0 |_{\k}, \ \ \lambda = \frac{t}{t+1} \in (0,1).
        \eeq


\section{Eschenburg and Bazaikin spaces}
\label{S:EschBaz}

Given $a = (a_1, a_2, a_3)$, $b = (b_1, b_2, b_3) \in \Z^3$, with $\sum a_i = \sum b_i$, recall that the Eschenburg biquotients (see \cite{AW}, \cite{Es}) are defined as $E^7_{a,b} := (\SU(3), \met) \bq S^1_{a,b}$, where $S^1_{a,b}$ acts isometrically on $(\SU(3), \met)$ via
$$
z \star A = 
\diag(z^{a_1}, z^{a_2}, z^{a_3}) \cdot A \cdot \diag(\bar z^{b_1}, \bar z^{b_2}, \bar z^{b_3}), \ \ A \in \SU(3), z \in S^1,
$$
and the left-invariant, right $\U(2)$-invariant metric $\met$ (with $\sec \geq 0$) on $\SU(3)$ is defined as in (\ref{met1}), where $\U(2) \hra \SU(3)$ via
$$
A \in \U(2) \mt \diag(\overline{\det(A)}, A) \in \SU(3).
$$
The action is free if and only if
        \beq
        \label{freeness}
            \gcd(a_1 - b_{\sigma(1)}, a_2 - b_{\sigma(2)}) = 1 \ \  \textrm{for all permutations} \ \ \sigma \in S_3,
        \eeq
in which case $E^7_{a,b}$ is called an \emph{Eschenburg space}.

It is important to remark that the above defined circle subgroup $S^1_{a,b}$ is not, in general, a subgroup of $\SU(3) \x \SU(3)$.  Indeed, $S^1_{a,b} \subset \SUU{3}{3} := \{(A,B) \in \U(3) \x \U(3) \mid \det A = \det B\}$.  This is not a problem, however, since the bi-invariant metric on $\SU(3)$ can be thought of as the restriction of the the analogously defined bi-invariant metric on $\U(3)$.  Hence, an element of $(A,B) \in \SUU{3}{3}$ maps $(\SU(3), \met_0)$ isometrically to itself via $X \in \SU(3) \mt A X B^{-1}$.  In particular, conjugation by an element of the centre of $\U(3)$ is an isometry (namely, the identity map) of $(\SU(3), \met_0)$, and remains an isometry with respect to the new metric $\met$.  Therefore the Eschenburg biquotient $E^7_{a',b'}$ defined by the action of the circle $S^1_{a',b'}$, where $a' = (a_1 + c, a_2 + c, a_3 + c)$ and $b' = (b_1 + c, b_2 + c, b_3 + c)$, with $c \in \Z$, is isometric to $E^7_{a,b}$.  Furthermore, introducing an ineffective kernel to the circle action will not alter the isometry class of the biquotient.  Thus $E^7_{\tilde a, \tilde b}$ defined by $\tilde a = (k a_1, k a_2, k a_3)$ and $\tilde b = (k b_1, k b_2, k b_3)$ is isometric to $E^7_{a,b}$.  In particular, it follows that a circle action by $S^1_{a,b} \subset \SUU{3}{3}$ can then be rewritten as the action of a circle subgroup of $\SU(3) \x \SU(3)$ via the change of parameters $(a_1, a_2, a_3, b_1, b_2, b_3) \mt (3 a_1 - \kappa, 3 a_2 - \kappa, 3 a_3 - \kappa, 3 b_1 - \kappa, 3 b_2 - \kappa, 3 b_3 - \kappa)$, where $\kappa := \sum a_i = \sum b_i$, without changing the isometry class.

From Section \ref{S:Biqs} it is clear that, for the $S^1_{a,b}$-action, permuting the $a_i$ (via the action of the Weyl group of $\SU(3)$) and permuting $b_2, b_3$ are isometries, while permuting all of the $b_i$ and swapping $a$, $b$ are diffeomorphisms.  Indeed, given our fixed choice of embedding $\U(2) \hra \SU(3)$, cyclic permutations of the $b_i$ (and, similarly, swapping $a$ and $b$ and considering cyclic permutations of the $a_i$) induce, in general, non-isometric metrics on the quotient $E^7_{a,b}$.  

It was shown in \cite{Es} that an Eschenburg biquotient $E^7_{a,b} = (\SU(3), \met) \bq S^1_{a,b}$ has positive curvature if and only if $b_i \not\in [\underline{a}, \overline{a}]$ for all $i=1,2,3$, where $\underline{a} := \min \{a_1, a_2, a_3 \}$, $\overline{a} := \max \{a_1, a_2, a_3 \}$.

\medskip
In order to define the Bazaikin spaces (see \cite{Baz}, \cite{Zi}, \cite{DE}), first let
$$ 
\syp(2) \cdot S^1_{q_1, \dots, q_5} = (\syp(2) \x S^1_{q_1, \dots, q_5})/\Z_2,\ \ \ \Z_2 = \{\pm(1, I)\},
$$
where $q_1, \dots, q_5 \in \Z$ and $\syp(2)$ is considered as a subgroup of $\SU(4)$ via the inclusion $\syp(2) \hra \SU(4)$ given by
\beq
\label{eq:embed}
A = S + T j \in \syp(2) \mt \hat A = \bpm S & T \\ - \bar T & \bar S \epm \in \SU(4), \ \ \ S, T \in M_2(\C).
\eeq
Equip $\SU(5)$ with a left-invariant and right $\U(4)$-invariant metric $\met$ as defined in (\ref{met1}), where $\U(4) \hra \SU(5)$ via $A \in \U(4) \mt \diag(\overline{\det A}, A) \in \SU(5)$.

Then $ \syp(2) \cdot S^1_{q_1, \dots, q_5}$ acts effectively and isometrically on $(\SU(5), \met)$ via
        $$[A,z] \star B = \diag(z^{q_1}, \dots, z^{q_5}) \cdot  B \cdot \diag(\bar z^q, \hat A),$$
with $q := \sum q_i$, $z \in S^1$, $B \in \SU(5)$, and $A \in \syp(2) \subset \SU(4)$.  The quotient $B^{13}_{q_1, \dots, q_5} := (\SU(5), \met) \bq \syp(2) \cdot S^1_{q_1, \dots, q_5}$ is called a Bazaikin biquotient. 

It is not difficult to show that the action of $\syp(2) \cdot S^1_{q_1, \dots, q_5}$ is free (hence $B^{13}_{q_1, \dots, q_5}$ is a \emph{Bazaikin space}) if and only all $q_1, \dots, q_5$ are odd and
        \beq
        \label{freeBaz}
            \gcd(q_{\sigma(1)} + q_{\sigma(2)}, q_{\sigma(3)} + q_{\sigma(4)}) = 2 \ \ \textrm{for all permutations} \ \ \sigma \in S_5.
        \eeq
From the discussion in Section \ref{S:Biqs} it follows that permuting the $q_i$ is an isometry of $B^{13}_{q_1, \dots, q_5}$.  Furthermore, it is well-known (see \cite{Zi}, \cite{DE}) that a general Bazaikin biquotient $B^{13}_{q_1, \dots, q_5}$ admits positive curvature if and only if $q_i + q_j > 0$ (or $<0$) for all $1 \leq i < j \leq  5$.


\section{Totally geodesic submanifolds of Bazaikin spaces}
\label{S:TGSubmBaz}

Consider the Lie group $\SU(5)$ equipped with the left-invariant, right $\U(4)$-invariant metric $\met$ as described in Section \ref{S:EschBaz}.  Let $\sigma : (\SU(5), \met) \to (\SU(5), \met)$ be the isometric involution defined by 
\beq
\label{eq:invol}
\sigma(A) = \diag(1,1,1,-1,-1)\cdot A \cdot \diag(1,1,1,-1,-1), \ \ \ A \in \SU(5).
\eeq
The fixed point set of this involution is totally geodesic in $(\SU(5), \met)$ and given by
$$
\SUU{3}{2} = \{(\diag(B,C) \in \U(3) \x \U(2) \mid \det B = \overline{\det C} \}.
$$
Given the embedding (\ref{eq:embed}) of $\syp(2)$ into $\SU(5)$, it follows that the intersection $\SUU{3}{2} \cap \syp(2)$ is $\Delta\U(2) := \{(\diag(1,C,\bar C) \mid C \in \U(2) \}$.

Hence the image of $\SUU{3}{2}$ in the quotient $\SU(5)/\syp(2)$ is the totally geodesic submanifold $\SUU{3}{2}/ \Delta\U(2)$.  Every element of $\SUU{3}{2}/ \Delta\U(2)$ has a unique representative of the form 
$\diag(B, I)$, where 
\beq
\label{eq:SU(3)}
B = X \left(\bsm 1 & \\ & Y^t \esm \right) \in \SU(3), \ \ \ 
\textrm{with } \ X \in \U(3), Y \in \U(2) \ \ 
\textrm{and } \det Y = \overline{\det X}.
\eeq
Therefore $\SUU{3}{2}/ \Delta\U(2)$ is a totally geodesic copy of $\SU(3)$ inside $\SU(5)/\syp(2)$.

The induced metric on $\SU(3) \cong \SUU{3}{2}/ \Delta\U(2) \subset \SU(5)/\syp(2)$ is invariant under $\SU(3) \x \U(2)$ and, after the identification (\ref{eq:SU(3)}), one may consider this metric as a left-invariant, right $\U(2)$-invariant metric on $\SU(3)$ of the form described in Section \ref{S:EschBaz}.  Indeed, the identity component of the total isometry group of $\SU(3) \cong \SUU{3}{2}/ \Delta\U(2)$ is given by
\begin{align*}
\Isom (\SU(3)) &= \left\{ \left( \bpm Z & \\ & W \epm,
                                      \bpm w & \\ & I \epm
                               \right) 
                                      \mid  (Z, W) \in U(3) \x U(2), w = (\det Z)(\det W)
                                      \right\}\\
&\cong \U(3) \x \U(2)
\end{align*}
and acts on $\SU(3) \cong \SUU{3}{2}/ \Delta\U(2)$ via
\begin{align}
\nonumber
\left(\bpm Z & \\ & W \epm,
\bpm w & \\ & I \epm \right) 
\star \left[\bpm B & \\ & I \epm \right] &= 
\left[\bpm Z B \left(\bsm \bar w & \\ & I \esm \right) & \\ & W \epm \right]\\
\label{eq:Isomgp}
&= \left[\bpm ZB \left(\bsm \bar w & \\ & W^t \esm \right) & \\ & I \epm \right]
\end{align}
where $[\diag(B,I)]  \in \SU(3)$, $(Z,W) \in \U(3) \x \U(2)$ and $w = (\det Z)(\det W)$.

Now, if the isometric left-action of $S^1_{q_1, \dots, q_5}$ on $\SU(5)/\syp(2)$ is free with ineffective kernel $\{\pm 1\}$, the same must be true of the induced action on the totally geodesic submanifold $\SU(3) \cong \SUU{3}{2}/ \Delta\U(2)$.  From (\ref{eq:Isomgp}) it is clear that the action on $\SU(3)$ is given by
$$
z \star B = 
\diag(z^{q_1}, z^{q_2}, z^{q_3}) B \diag(\bar z^{q}, z^{q_4}, z^{q_5}), \ \ B \in \SU(3), z \in S^1,
$$
where $q = \sum q_i$.  The quotient is an Eschenburg space $\SU(3) \bq S^1_{q_1, \dots, q_5}$ totally geodesically embedded in the Bazaikin space $B^{13}_{q_1, \dots, q_5}$.  The circle action defining the Eschenburg space can be made effective.  Indeed, since the $q_i$ are odd and from the discussion of reparameterisations in Section \ref{S:EschBaz}, it follows that the Eschenburg space $\SU(3) \bq S^1_{q_1, \dots, q_5}$ is isometric to the Eschenburg space $E^7_{a,b}$, where $a = (a_1, a_2, a_3) = (\tfrac{1}{2}(q_1 - 1), \tfrac{1}{2}(q_2 - 1), \tfrac{1}{2}(q_3 - 1))$ and $b = (b_1, b_2, b_3) = (\tfrac{1}{2}(q - 1), - \tfrac{1}{2}(q_4 + 1),- \tfrac{1}{2}(q_5 + 1))$.

It is now simple to recover the remaining nine totally geodesic, embedded Eschenburg spaces in $B^{13}_{q_1, \dots, q_5}$.  As every permutation $\sigma \in S_5$ of the $q_i$ is an isometry, it follows that $B^{13}_{q_1, \dots, q_5} = B^{13}_{q_{\sigma(1)}, \dots, q_{\sigma(5)}}$ and hence, by the same reasoning as before, the Eschenburg space $E^7_{a_\sigma,b_\sigma}$, where $a_\sigma = (\tfrac{1}{2}(q_{\sigma(1)} - 1), \tfrac{1}{2}(q_{\sigma(2)} - 1), \tfrac{1}{2}(q_{\sigma(3)} - 1))$ and $b_\sigma = (\tfrac{1}{2}(q - 1), - \tfrac{1}{2}(q_{\sigma(4)} + 1),- \tfrac{1}{2}(q_{\sigma(5)} + 1))$, is a totally geodesic, embedded submanifold.  That generically there are ten such submanifolds follows since permutations of the entries of $a_\sigma$, and of the last two entries of $b_\sigma$, are isometries.  This, of course, is equivalent to fixing the order of the $q_i$ and permuting the signs of the entries along the diagonal in the involution (\ref{eq:invol}) acting on $\SU(5)$, thus achieving ten (isometric) copies of $\SU(3)$ which quotient to the desired Eschenburg spaces (see \cite{DE}).


\section{Totally geodesic embeddings of Eschenburg spaces}
\label{S:Converse}

In Section \ref{S:TGSubmBaz} it was shown that every Bazaikin space contains a totally geodesically embedded Eschenburg space.  The converse statement, namely that every Eschenburg space can be totally geodesically embedded into a Bazaikin space, is also true.  Indeed, it will now be shown that every Eschenburg space can be totally geodesically embedded into infinitely many Bazaikin spaces.

Let $E^7_{a,b}$ be the Eschenburg space given by $a = (a_1, a_2, a_3)$, $b = (b_1, b_2, b_3) \in \Z^3$ with $\sum a_i = \sum b_i$ and satisfying the freeness condition (\ref{freeness}):
$$
\gcd(a_1 - b_{\sigma(1)}, a_2 - b_{\sigma(2)}) = 1 \ \  \textrm{for all permutations} \ \ \sigma \in S_3.
$$
By the discussion in Section \ref{S:TGSubmBaz}, a candidate for a Bazaikin space into which to embed is given by the $5$-tuple $(q_1, \dots, q_5) := (2 a_1 + 1, 2 a_2 + 1, 2 a_3 + 1, -(2 b_2 + 1), -(2 b_3 + 1))$, with $q = \sum q_i = 2 b_1 + 1$.
As each of the $q_i$ is odd, one need only check the condition (\ref{freeBaz}):
$$
\gcd(q_{\sigma(1)} + q_{\sigma(2)}, q_{\sigma(3)} + q_{\sigma(4)}) = 2 \ \ \textrm{for all permutations} \ \ \sigma \in S_5.
$$
It follows from the Eschenburg freeness condition (\ref{freeness}) (and from $\sum a_i = \sum b_i$) that $\gcd(q_4 + q_i, q_5 + q_j) = 2$ for all $i \neq j \in \{1,2,3\}$.  Furthermore, if $\{i,j,k\} = \{1,2,3\}$ and $\{\ell,m\} = \{4,5\}$, then since $q = \sum q_i$ it follows that
\begin{align*}
\gcd(q_i + q_j, q_k + q_\ell) &= \gcd(q_i + q_j, q - q_m)\\
                             &= \gcd(q - q_m, q_k + q_\ell)\\
\textrm{and that }\ \ \gcd(q_i + q_j, q_4 + q_5) &= \gcd(q_i + q_j, q - q_k).
\end{align*}

Hence the $5$-tuple $(q_1, \dots, q_5) = (2 a_1 + 1, 2 a_2 + 1, 2 a_3 + 1, -(2 b_2 + 1), -(2 b_3 + 1))$ defines a Bazaikin space if and only if
\beq
\label{eq:NonSing}
d_{k \ell} := \gcd(a_i + a_j + 1, a_k - b_\ell) = 1 \ \ \textrm{for all} \ \ \ell \in \{i,j,k\} = \{1,2,3\}.
\eeq

\begin{lem}
\label{lem:oddprimes}
Let $a = (a_1, a_2, a_3), b = (b_1, b_2, b_3) \in \Z^3$ satisfy (\ref{freeness}) and $\sum a_i = \sum b_i$.  Suppose that $p$ is a prime divisor of $\gcd(a_i + a_j + 1, a_k - b_\ell)$, where $\ell \in \{i,j,k\} = \{1,2,3\}$.  Then $p$ is odd.
\end{lem}

\begin{proof}
By (\ref{freeness}) together with $\sum a_i = \sum b_i$ it follows that either all $a_i$ or all $b_i$ have the same parity.  Hence either all $a_i + a_j + 1$ or all $b_m + b_n + 1$ are odd, where $1 \leq i < j \leq 3$, $1 \leq m < n \leq 3$.  Since $\gcd(a_i + a_j + 1, a_k - b_\ell) = \gcd(b_m + b_n + 1, a_k - b_\ell)$, where $\{i,j,k\} = \{\ell,m,n\} = \{1,2,3\}$, the conclusion follows.
\end{proof}

Recall now that one can (isometrically) rewrite the Eschenburg space $E^7_{a,b}$ as $E^7_{a_c,b_c}$, where $a_c = (a_1 + c, a_2 + c, a_3 + c)$ and $b_c = (b_1 + c, b_2 + c, b_3 + c)$, for any $c \in \Z$.  In this case the candidate Bazaikin biquotient is given by the $5$-tuple $(q^c_1, \dots, q^c_5) := (2 (a_1 + c) + 1, 2 (a_2 + c) + 1, 2 (a_3 + c) + 1, -(2 (b_2 + c) + 1), -(2 (b_3 + c) + 1))$, with $q^c = \sum q^c_i = 2 (b_1 + c) + 1$.  Analogously to above, this Bazaikin biquotient is non-singular if and only if 
\beq
\label{eq:modNonSing}
d^c_{k \ell} := \gcd(a_i + a_j + 1 + 2c, a_k - b_\ell) = 1 \ \ \textrm{for all} \ \ \ell \in \{i,j,k\} = \{1,2,3\}.
\eeq

\begin{lem}
\label{lem:free}
Let $a = (a_1, a_2, a_3), b = (b_1, b_2, b_3) \in \Z^3$ satisfy (\ref{freeness}) and $\sum a_i = \sum b_i$.  For all $\ell \in \{i,j,k\} = \{1,2,3\}$, let $p_{k\ell 1}, \dots, p_{k\ell r_{k \ell}}$ be the prime divisors, if any, of $a_k - b_\ell$ which satisfy $\gcd(p_{k \ell \tau}, a_i + a_j + 1) = 1$, for all $1 \leq \tau \leq r_{k \ell}$.  Set 
\beq
\label{eq:cvalue}
c = c_\mu := \pm 2^{\mu - 1} \left(\prod_{k, \ell = 1}^3 \prod_{\tau = 1}^{r_{k \ell}} p_{k \ell \tau} \right)^\mu,
\eeq
where $\mu$ is an arbitrary positive integer.  Then $d^c_{k \ell} = \gcd(a_i + a_j + 1 + 2c, a_k - b_\ell) = 1$.
\end{lem}

\begin{proof}
Suppose that $p$ is a prime divisor of $d^c_{k \ell}$.  As $p$ divides $a_k - b_\ell$, then $p$ divides either $a_i + a_j + 1$ or $c$.  But $p$ divides $a_i + a_j + 1 + 2c$, hence must divide both $a_i + a_j + 1$ and $c$, and furthermore must be odd by Lemma \ref{lem:oddprimes}.  Now, by definition of $c$, $p$ must divide $a_t - b_u$, for some $(t,u) \neq (k, \ell)$, and must satisfy $\gcd(p, a_r + a_s + 1) = 1$, where $\{r,s,t\} = \{1,2,3\}$.  However, because of the freeness condition (\ref{freeness}) for Eschenburg spaces, either $t = k$ or $u = \ell$.  If $t = k$, then $1 = \gcd(p, a_r + a_s + 1) = \gcd(p, a_i + a_j + 1) = p$.  On the other hand, if $u = \ell$ then $k \neq t \in \{1,2,3\}$, that is, it may assumed without loss of generality that $(t,u) = (i,\ell)$.  As $p$ divides $a_i + a_j + 1$, $a_k - b_\ell$ and $a_i - b_\ell$, it follows that $1 = \gcd(p, a_r + a_s + 1) = \gcd(p, a_j + a_k + 1) = \gcd(p, (a_i + a_j + 1) + (a_k - b_\ell) - (a_i - b_\ell)) = p$.
\end{proof}

\begin{cor}
\label{cor:embed}
With notation as above, for all $c = c_\mu$, $\mu \in \Z$, $\mu  > 0$, the Bazaikin biquotient $B^{13}_{q^c_1, \dots, q^c_5}$ is non-singular and contains $E^7_{a,b} = E^7_{a_c,b_c}$ as a totally geodesic submanifold.
\end{cor}

\begin{rem}
The conditions (\ref{eq:NonSing}) and (\ref{freeness}) ensuring that the Bazaikin biquotient is non-singular are precisely the conditions which ensure that all ten of the totally geodesic, embedded Eschenburg biquotients are themselves non-singular.  The equivalence of the non-singularity of the Bazaikin biquotient and these ten submanifolds was already observed in \cite{DE}.
\end{rem}

\begin{rem}
Given a Bazaikin space $B^{13}_{q_1, \dots, q_5}$, with $(q_1, \dots, q_5) = (2 a_1 + 1, 2 a_2 + 1, 2 a_3 + 1, -(2 b_2 + 1), -(2 b_3 + 1))$ and $q = \sum q_i = 2 b_1 + 1$, into which the Eschenburg space $E^7_{a, b}$ has a totally geodesic embedding, that is, condition (\ref{eq:NonSing}) holds, then the diffeomorphic, but in general non-isometric, Eschenburg space $E^7_{b,a}$ totally geodesically embeds into the Bazaikin space $B^{13}_{q, -q_4, -q_5, -q_2, -q_3}$.  This follows easily from condition (\ref{eq:NonSing}) together with $\sum a_i = \sum b_i$.  The Bazaikin spaces $B^{13}_{q_1, \dots, q_5}$ and $B^{13}_{q, -q_4, -q_5, -q_2, -q_3}$ are diffeomorphic but not isometric in general (see \cite{EKS}, \cite{FZ}).
\end{rem}

The integral cohomology rings of Eschenburg spaces and Bazaikin spaces are well-known (see \cite{Es2}, \cite{CEZ}, \cite{Baz}, \cite{FZ}).  In particular, if $\sigma_i (x) := \sigma_i (x_1, \dots, x_m)$ denotes the $i^{\rm th}$ symmetric polynomial in $x = (x_1, \dots, x_m)$, then $H^4(E^7_{a,b}) = \Z_r$, where $r = |\sigma_2(a) - \sigma_2(b)|$ is always odd \cite[Remark 1.3]{Kr}, and $H^6(B^{13}_{q_1, \dots, q_5}) = \Z_s$, where $s = \tfrac{1}{8}|\sigma_3(q_1, q_2, q_3, q_4, q_5, -\sum q_i)|$. 

\begin{lem}
\label{lem:order}
Let $a:=(a_1, a_2, a_3), b:=(b_1, b_2, b_3) \in \Z^3$, with $\sum a_i = \sum b_i$, and for $c \in \Z$ define the $6$-tuple
$$
q_c := (2(a_1 + c) + 1, 2(a_2 + c) + 1, 2(a_3 + c) + 1, -2(b_1 + c) - 1, -2(b_2 + c) - 1, -2(b_3 + c) - 1).
$$
If $c \neq d \in \Z$, then $|\sigma_3(q_c)| = |\sigma_3(q_d)|$ if and only if either $\sigma_2(a) = \sigma_2(b)$ or 
$$
c+d = \frac{\sigma_3(a)-\sigma_3(b)}{\sigma_2(a) - \sigma_2(b)} - \sigma_1(a) - 1.
$$
\end{lem}

\begin{proof}
Recall that $\sigma_i(x)$, the $i^{\rm th}$ symmetric polynomial in $x = (x_1, \dots, x_m)$, is defined as the coefficient of $y^{m-i}$ in the product $\prod_{j=1}^m (y + x_j)$.  Then, with $y_+ = y + 1 + 2c$ and $y_- = y - 1 - 2c$, $\sigma_3 (q_c)$ is given by the coefficient of $y^3$ in the product
\begin{align*}
\prod_{j=1}^3 (y + 2a_j + 1 + 2c) \prod_{j=1}^3 (y - 2b_j - 1 - 2c)
&= \prod_{j=1}^3 (y_+ + 2a_j) \prod_{j=1}^3 (y_- -2 b _j) \\
&= \left(\sum_{j = 0}^3 \sigma_j (a) y_+^{3-j} \right) \left(\sum_{j = 0}^3 \sigma_j (b) y_-^{3-j} \right).
\end{align*}
Hence, since $\sigma_1(a) = \sum a_j = \sum b_j = \sigma_1(b)$,
$$
\sigma_3 (q_c) = 8(\sigma_3(a)-\sigma_3(b)) - 8(\sigma_1(a) + 2c+1)(\sigma_2(a)-\sigma_2(b)).
$$
Now $|\sigma_3(q_c)| = |\sigma_3(q_d)|$ if and only if $\sigma_3(q_c) = \pm \sigma_3(q_d)$, that is, if and only if either 
\begin{align*}
(c - d)(\sigma_2(a)-\sigma_2(b)) &= 0 \ \ \textrm{ or} \\
(c + d)(\sigma_2(a)-\sigma_2(b)) &= (\sigma_3(a)-\sigma_3(b)) - (\sigma_1(a) + 1)(\sigma_2(a)-\sigma_2(b)),
\end{align*}
from which the claim follows.
\end{proof}

\begin{cor}
\label{cor:topology}
The collection of Bazaikin spaces into which a particular Eschenburg space can be totally geodesically embedded consists of infinitely many distinct homotopy types.
\end{cor}

\begin{proof}
Given an Eschenburg space $E^7_{a,b}$ and $c$ of the form (\ref{eq:cvalue}), it has been shown that there is a totally geodesic embedding of $E^7_{a,b}$ into each Bazaikin space $B^{13}_{q^c_1, \dots, q^c_5}$, where $(q^c_1, \dots, q^c_5) = (2 a_1 + 2c + 1, 2 a_2 + 2c + 1, 2 a_3 + 2c + 1, -2 b_2 - 2c - 1, -2 b_3 -2 c - 1)$.  Fix one such value of $c$, that is, fix some $\mu > 0$.  Since $2^{\mu-1} > 1$ for $\mu > 1$, we may assume without loss of generality that $|c| > 1$.  Let $c_1 = c^{\alpha_1}$ and $c_2 = c^{\alpha_2}$ for some $0 < \alpha_1 < \alpha_2$.  With the same notation as in Lemma \ref{lem:order}, the order of $H^6$ for the Bazaikin space corresponding to $c_i$ is given by $s_i = \tfrac{1}{8} |\sigma_3(q_{c_i})|$, for $i = 1,2$ respectively.  By Lemma \ref{lem:order}, $s_1 = s_2$ if and only if either $\sigma_2(a) - \sigma_2(b) = 0$ or 
\beq
\label{eq:diffcohom}
c_1 + c_2 = \frac{\sigma_3(a)-\sigma_3(b)}{\sigma_2(a) - \sigma_2(b)} - \sigma_1(a) - 1.
\eeq
But, by \cite[Remark 1.3]{Kr}, $\sigma_2(a) - \sigma_2(b)$ is odd, hence non-zero.  Therefore (\ref{eq:diffcohom}) must hold.  However, as the right-hand side is constant, it is clear that (\ref{eq:diffcohom}) can hold for at most one pair $0 < \alpha_1 < \alpha_2$.
\end{proof}

The following example shows that the $c$ values given in Lemma \ref{lem:free} do not achieve all Bazaikin spaces into which an Eschenburg space can be totally geodesically embedded.

\begin{eg}
\label{eg:cohom2}
Consider the positively curved Eschenburg space $E^7_{a,b}$ given by $a = (2,0,0)$ and $b = (15, -2, -11)$.  The corresponding candidate for a Bazaikin space is given by $q = (5, 1, 1, 3, 21)$, which is singular because $\gcd(q_1 + q_2, q_4 + q_5) = 6$.

By the expression in (\ref{eq:cvalue}), $E^7_{a,b}$ can be totally geodesically embedded into each of the Bazaikin spaces given by $(5 + 2c, 1 + 2c, 1 + 2c, 3 - 2c, 21 - 2c)$, with $c = \pm 2^{\mu - 1}(2^3 \cdot 5^2 \cdot 11^2 \cdot 13^2)^\mu$, for $\mu > 0$.  None of these spaces is positively curved, since it is clear that any value of $\mu > 0$ leads to $q_i + q_j$ of mixed signs.

On the other hand, note that when $c = -1$ the Eschenburg space $E^7_{a,b}$ can be rewritten as $E^7_{\tilde a, \tilde b}$, with $\tilde a = (1, -1, -1)$ and $\tilde b = (14, -3, -12)$, and the corresponding Bazaikin biquotient $B^{13}_{\tilde q}$ with $\tilde q = (3, -1, -1, 5, 23)$ is non-singular. Note that $B^{13}_{\tilde q}$ does not have positive curvature.

However, when $c = 2$ and when $c = 5$, the resulting biquotients are both non-singular and positively curved, namely the Bazaikin spaces $B^{13}_{q'}$, with $q' = (9, 5, 5, -1, 17)$ (corresponding to $a' = (4, 2, 2)$, $b' = (17, 0, -9)$), and $B^{13}_{q''}$, with $q'' = (15, 11, 11, -7, 11)$ (corresponding to $a' = (7, 5, 5)$, $b' = (20, 3, -6)$).

Finally, notice that the order of $H^6$ is $503$, $1541$ and $2579$ for $B^{13}_{\tilde q}$, $B^{13}_{q'}$ and $B^{13}_{q''}$ respectively, hence these spaces are not even homotopy equivalent.
\end{eg}


\section{Embeddings inducing positive sectional curvature}
\label{S:curv}

Given that every Eschenburg space can be totally geodesically embedded into infinitely many Bazaikin spaces, it is natural to ask whether a positively curved Eschenburg space admits a totally geodesic embedding into a positively curved Bazaikin space.

Recall that an Eschenburg space $E^7_{a,b}$ has positive curvature if and only if $b_i \not\in [\underline{a}, \overline{a}]$ for all $i=1,2,3$, where $\underline{a} := \min \{a_1, a_2, a_3 \}$, $\overline{a} := \max \{a_1, a_2, a_3 \}$.  Because $\sum a_j = \sum b_j$, this is equivalent to the requirement that two of the $b_i$ lie on one side of $[\underline{a}, \overline{a}]$, and one on the other.  Indeed, given the metric on $E^7_{a,b}$ defined in Section \ref{S:EschBaz}, it turns out that $b_2$ and $b_3$ must lie on the same side of $[\underline{a}, \overline{a}]$, see \cite{Es}, \cite{Ke}.  Since permuting the $a_i$ and permuting $b_2$ and $b_3$ are isometries, the condition for positive curvature is, after relabelling if necessary, equivalent to
\beq
\label{poscond}
            b_3 \leq b_2 < a_3 \leq a_2 \leq a_1 < b_1 \ \ 
\textrm{ or } \ \ 
            b_1 < a_1 \leq a_2 \leq a_3 < b_2 \leq b_3.
\eeq
In fact, the second chain of inqualities is equivalent to the first via the reparametrization of $S^1_{a,b}$ via $z \mt \bar z$, and therefore one may restrict attention to the first chain.  

Now $E^7_{a,b}$ can be totally geodesically embedded into the (possibly singular) Bazaikin biquotient $B^{13}_{q_1, \dots, q_5}$, with $(q_1, \dots, q_5) = (2 a_1 + 1, 2 a_2 + 1, 2 a_3 + 1, -(2 b_2 + 1), -(2 b_3 + 1))$.  As discussed in Section \ref{S:EschBaz}, $B^{13}_{q_1, \dots, q_5}$ has positive curvature if and only if $q_i + q_j > 0$ (or $<0$) for all $1 \leq i < j \leq  5$.  Since the first chain of inequalities in (\ref{poscond}) implies that $q_i + q_j = 2(a_i - b_{j-2}) > 0$, for all $i = 1,2,3$, $j = 4,5$, it follows that $B^{13}_{q_1, \dots, q_5}$ has positive curvature if $a_i + a_j + 1 > 0$, for all $1 \leq i < j \leq 3$, and $b_2 + b_3 + 1 < 0$.  But $a_3 \leq a_2 \leq a_1$, hence $B^{13}_{q_1, \dots, q_5}$ has positive curvature if and only if
\beq
\label{eq:Bazcurv}
b_2 + b_3 + 1 < 0 < a_2 + a_3 + 1.
\eeq
In particular, it is necessary that $0 \leq a_2 \leq a_1 < b_1$ and $b_3 \leq -1$.  Clearly the special case $b_2 < 0 \leq a_3$ ensures that the inequalities in (\ref{eq:Bazcurv}) are satisfied.  

If $B^{13}_{q_1, \dots, q_5}$ does not have positive curvature, then one can find a $c \in \Z$ such that the Bazaikin biquotient $B^{13}_{q^c_1, \dots, q^c_5}$, with 
$$
(q^c_1, \dots, q^c_5) := 
(2 (a_1 + c) + 1, 2 (a_2 + c) + 1, 2 (a_3 + c) + 1, -(2 (b_2 + c) + 1), -(2 (b_3 + c) + 1)),
$$
is positively curved.  Indeed, $B^{13}_{q^c_1, \dots, q^c_5}$ has positive curvature if and only if $c \in Z$ satisfies
\beq
\label{eq:Bazcurvmod}
-\frac{1}{2}(a_2 + a_3 + 1) < c < -\frac{1}{2}(b_2 + b_3 + 1).
\eeq
It then remains only to examine the values of $c$ given by (\ref{eq:Bazcurvmod}) to determine which, if any, of the Bazaikin biquotients $B^{13}_{q^c_1, \dots, q^c_5}$ are non-singular.

One particular example, where an embedding into a positively curved Bazaikin space exists, is that of a positively curved Eschenburg space of cohomogeneity-one, given by $a = (p, 1, 1)$, $b = (p+2, 0, 0)$, with $p \geq 1$.  Here (\ref{eq:Bazcurvmod}) yields $-1 \leq c \leq 0$.  Indeed, $c = -1$ ensures that $B^{13}_{q^c_1, \dots, q^c_5}$ is non-singular and positively curved.  This was first observed by W.\ Ziller and appeared in \cite{DE}.

Note that in this example $c = - a_3 = -1$.  Indeed, for an arbitrary Eschenburg space $E^7_{a,b}$, the choice $c = -a_3$ ensures $B^{13}_{q^c_1, \dots, q^c_5}$ has positive curvature, although in general one cannot hope that this space will be non-singular.  Example \ref{eg:cohom2} illustrates this phenomenon and, furthermore, that the values of $c$ suggested by the expression (\ref{eq:cvalue}) are not of much use when it comes to finding a Bazaikin space of positive curvature.  

It turns out, in fact, that there are examples of positively curved Eschenburg spaces for which none of the values of $c$ coming from (\ref{eq:Bazcurvmod}) yield a non-singular Bazaikin biquotient, that is, a totally geodesic embedding constructed as in Section \ref{S:Converse} cannot be into a positively curved Bazaikin space.  A list of such examples is given in Table \ref{tab:fail}.  Notice, in particular, that the first and last examples in Table \ref{tab:fail} are Eschenburg spaces of cohomogeneity-two (see \cite{GSZ}).  In fact, there are infinitely many such cohomogeneity-two examples, for example, the infinite families given by $a = (15015k + 39, 0, 0)$, $b = (15015k + 55, -3, -13)$, $k \geq 0$, and $a = (15015k + 12909, 0, 0)$, $b = (15015k + 12925, -3, -13)$, $k \geq 0$, respectively.

\begin{rem}
\label{rem:reduce}
One might hope that introducing an ineffective kernel into the $S^1_{a,b}$ action defining a positively curved Eschenburg space $E^7_{a,b}$ would yield a Bazaikin space with positive curvature into which it can be totally geodesically embedded.  However, such a modification reduces to the case discussed above.  Indeed, if $\tilde a = (\lambda a_1 + d,\lambda a_2 + d, \lambda a_3 + d)$ and $\tilde b = (\lambda b_1 + d, \lambda b_2 + d, \lambda b_3 + d)$, where without loss of generality $\lambda > 0$, then, by (\ref{freeness}), $\gcd(\tilde a_i - \tilde b_\ell, \tilde a_j - \tilde b_m) = \lambda$.  Hence the candidate for a positively curved Bazaikin space $B^{13}_{\tilde q}$ is given by $\tilde q = (2 \tilde a_1 + 1, 2 \tilde a_2 + 1, 2 \tilde a_3 + 1, -(2 \tilde b_2 + 1), -(2 \tilde b_3 + 1))$, where $\gcd(\tilde a_i + \tilde a_j + 1, \tilde a_k - \tilde b_\ell) = \lambda$, for all $\ell \in \{i,j,k\} = \{1,2,3\}$, and $\tilde a_2 + \tilde a_3 + 1 > 0 > \tilde b_2 + \tilde b_3 + 1$.  Consequently $\lambda$ must be odd and there is $c \in \Z$ such that $2 d + 1 = \lambda (2 c + 1)$, in which case $B^{13}_{\tilde q}$ is non-singular if and only if (\ref{freeness}) and (\ref{eq:modNonSing}) hold, and has positive curvature if and only if (\ref{eq:Bazcurvmod}) is satisfied.
\end{rem}

\begin{table}[!ht]
\begin{tabular}{|c|c|c|c|} \hline
\multicolumn{2}{|c|}{$\mathbf{ E^7_{a,b}}$} & \multicolumn{2}{|c|}{$\mathbf{ B^{13}_{q^c_1, \dots, q^c_5}}$}\\ \hline 
$\mathbf{ a}$ & $\mathbf{ b}$ & $\mathbf{ (q^c_1, \dots, q^c_5)}$ & $\mathbf{ \sec > 0}$ \\ \hline \hline

$(39, 0, 0)$ & $(55, -3, -13)$ & $(79 + 2c, 1 + 2c, 1 + 2c, 5 - 2c, 25 - 2c)$ & $0 \leq c \leq 7$ \\ \hline
$(77, 2, 0)$ & $(93, -3, -11)$ & $(155 + 2c, 5 + 2c, 1 + 2c, 5 - 2c, 21 - 2c)$ & $-1 \leq c \leq 6$ \\ \hline
$(171, 2, 0)$ & $(187, -3, -11)$ & $(343 + 2c, 5 + 2c, 1 + 2c, 5 - 2c, 21 - 2c)$ & $-1 \leq c \leq 6$ \\ \hline
$(225, 4, 0)$ & $(247, -5, -13)$ & $(451 + 2c, 9 + 2c, 1 + 2c, 9 - 2c, 25 - 2c)$ & $-2 \leq c \leq 8$ \\ \hline
$(281, 3, 0)$ & $(294, -2, -8)$ & $(563 + 2c, 7 + 2c, 1 + 2c, 3 - 2c, 15 - 2c)$ & $-1 \leq c \leq 4$ \\ \hline
$(309, 6, 0)$ & $(323, -3, -5)$ & $(619 + 2c, 13 + 2c, 1 + 2c, 5 - 2c, 9 - 2c)$ & $-3 \leq c \leq 3$ \\ \hline
$(664, 2, 0)$ & $(678, -3, -9)$ & $(1329 + 2c, 5 + 2c, 1 + 2c, 5 - 2c, 17 - 2c)$ & $-1 \leq c \leq 5$ \\ \hline
$(827, 4, 0)$ & $(843, -3, -9)$ & $(1655 + 2c, 9 + 2c, 1 + 2c, 5 - 2c, 17 - 2c)$ & $-2 \leq c \leq 5$ \\ \hline
$(12909, 0, 0)$ & $(12925, -3, -13)$ & $(25819 + 2c, 1 + 2c, 1 + 2c, 5 - 2c, 25 - 2c)$ & $0 \leq c \leq 7$ \\ \hline
\end{tabular}
\vspace{.1cm}
\caption{Eschenburg spaces which do not totally geodesically embed into a positively curved, non-singular Bazaikin space via the construction of Section \ref{S:Converse}.}
\label{tab:fail}
\end{table}


\end{document}